\documentclass[11pt,leqno]{amsart}
\usepackage{amssymb,amsmath,amsthm,soul,color}
\usepackage[normalem]{ulem}
\usepackage{amssymb,amsmath,amsthm}
\usepackage{mathrsfs}
\usepackage[utf8]{inputenc}
\usepackage{todonotes}
\usepackage[left=3cm, right=3cm, top=2.5cm, bottom=2.5cm]{geometry}
\usepackage{url}
\usepackage[numbers,sort&compress]{natbib}
\usepackage{fixmath}

\usepackage[T1]{fontenc}
\usepackage{microtype}

\usepackage[colorlinks=true,urlcolor=blue,
citecolor=red,linkcolor=blue,linktocpage,pdfpagelabels,
bookmarksnumbered,bookmarksopen]{hyperref}

\usepackage{xcolor,cancel}

\newtheorem{Th}{Theorem}[section]

\newtheorem{Lem}[Th]{Lemma}

\newcommand{\eps}{\varepsilon}

\newcommand{\R}{\mathbb{R}}

\newcommand{\Z}{\mathbb{Z}}

\newcommand{\cC}{{\mathcal C}}
\newcommand{\cD}{{\mathcal D}}

\newcommand{\cI}{{\mathcal I}}
\newcommand{\cJ}{{\mathcal J}}

\newcommand{\cN}{{\mathcal N}}

\newcommand{\weakto}{\rightharpoonup}

\numberwithin{equation}{section}

\DeclareMathOperator*{\essinf}{ess\,inf}

\newcommand{\supp}{\mathrm{supp}\,}

\newcommand{\loc}{\mathrm{loc}}

\begin{document}
\title{Nonlinear Schr\"odinger equations with a critical, inverse-square potential}

\author[B. Bieganowski]{Bartosz Bieganowski}
\author[A. Konysz]{Adam Konysz}
\author[S. Secchi]{Simone Secchi}

\address[B. Bieganowski]{\newline\indent
	Faculty of Mathematics, Informatics and Mechanics, \newline\indent
	University of Warsaw, \newline\indent
	ul. Banacha 2, 02-097 Warsaw, Poland}
\email{\href{mailto:bartoszb@mimuw.edu.pl}{bartoszb@mimuw.edu.pl}}

\address[A. Konysz]{\newline\indent
	Faculty of Mathematics and Computer Science, \newline\indent
	Nicolaus Copernicus University, \newline\indent
	ul. Chopina 12/18, 87-100 Toruń, Poland}
\email{\href{mailto:adamkon@mat.umk.pl}{adamkon@mat.umk.pl}}

\address[S. Secchi]{\newline\indent
	Dipartimento di Matematica e Applicazioni, \newline\indent
	Università degli Studi di Milano Bicocca, \newline\indent
	via R. Cozzi 55, I-20125 Milano, Italy}
\email{\href{mailto:simone.secchi@unimib.it}{simone.secchi@unimib.it}}

\begin{abstract}
    We study the existence of solutions of the following nonlinear Schr\"odinger equation
$$
-\Delta u+V(x)u-\frac{(N-2)^2}{4|x|^2}u=f(x,u)
$$
where $V:\R^N\to\R$ and $f:\R^N\times \R\to \R$ are periodic with respect to $x\in\R^N.$ We assume that $V$ has positive essential infimum, $f$ satisfies weak growth conditions and $N\geq 3.$ The approach to the problem uses variational methods with nonstandard functional setting. We obtain the existence of the ground state solution using the new profile decomposition.
   
\end{abstract}

\maketitle

\section{Introduction}
In this paper we are interested in finding so called {\em standing waves} for nonlinear Schr\"odinger equation with inverse-square potential with critical coefficient:
\begin{equation}\label{eq:Schrodinger}
    i\frac{\partial\Psi}{\partial t}=-\Delta_x\Psi +\left(V(x)+\lambda\right)\Psi - \frac{(N-2)^2}{4|x|^2}\Psi -g(x,|\Psi|)\Psi,\quad (t,x)\in\R\times(\R^N \setminus \{0\}),
\end{equation}
where $i^2=-1$ and $N\geq 3.$

Looking for solution of the form $\Psi(t,x)=e^{-i\lambda t}u(x)$ leads us to the following equation:
\begin{equation}\label{eq:main}
    -\Delta u  +V(x)u - \frac{(N-2)^2}{4|x|^2}u =f(x,u), \quad x \in \R^N \setminus \{0\}
\end{equation}
where we assume that $V\colon \R^N\to\R$ and $f:\R^N\times\R$ are periodic with respect to $x\in\R^N$ and $N\geq 3.$ Moreover we assume that $f$ has superlinear and subcritical growth. The Schr\"odinger equation plays an important role in many physical models. For example, in nonlinear optics it describes the propagation of light through optical structures with periodic structure (such as photonic crystals). In the presence of defects in the material, the inverse-square potential (Hardy-type) $\frac{\mu}{|x|^2}$ appears in the equation. Such singular potentials also arise in many other areas as quantum mechanics, nuclear and molecular physics and even quantum cosmology. The reader who would like to know more about physical motivations of Hardy-type potentials may check \cite{FrankLand,FelliTerracini, FelliMarchini} and references therein.

In the mathematical perspective the inverse-square potentials were widely investigated by many authors. In most of the works the Hardy type potential is controlled by a constant $\mu<\frac{(N-2)^2}{4}.$ For instance such subcritical problem was investigated by Guo and Mederski in \cite{GuoMederski} in a strongly indefinite case. Other exemplary contributions with subcritical nonlinearity are due to Li, Li and Tang in \cite{LiLiTang} with Berestycki-Lions type conditions, as well as Zhang and Zhang in \cite{Zhang} for the system of Schr\"odinger equations. For critical nonlinearity, i.e., nonlinearity with growth with power equal to the critical Sobolev exponent $2^*:=\frac{2N}{N-2}$ there are results by Deng, Jin, Peng \cite{DengJinPeng}, Felli and Pistoia \cite{FelliPistoia}, Smets in \cite{Smets} where he obtained nonexistence of solutions, and Terracini in \cite{Terracini}. Critical case when - as in \eqref{eq:main} - the critical constant $\mu=\frac{(N-2)^2}{4}$ introduces many additional difficulties. The natural norm for our problem is not equivalent to the classical one in $H^1(\R^N),$ which motivates the definition of the space $X^1(\R^N)$ to be the completion of $H^1(\R^N)$ under the natural norm to our problem. A crucial difficulty is that $X^1(\R^N)$ is not invariant under the translations, which makes compactness arguments significantly more delicate. The above mentioned nonstandard functional space embeds continuously into $H^s(\R^N)$ for every $s\in (0,1).$ This embedding is due to the inequality proven by Frank in \cite[Theorem 1.2]{Frank}, where the space in the context of partial differential equations with Hardy-type potentials was firstly defined independently by Suzuki in \cite{Suzuki} and by Trachanas and Zographopulos in \cite{TrachanasZographopoulos}. Later, Mukherjee, Nam, and Nguyen also defined and used this space in \cite{Tai}. Most recently, the Berestycki–Lions scalar-field equation in this setting was studied by Bieganowski and Strzelecki \cite{BBDS}.

In this paper the nonlinearity satisfies superlinear and subcritical growth conditions, see (F1)-(F3) below. The idea for generating Cerami sequences is based on the use of \cite[Theorem 5.1]{BB}, in which the Nehari manifold appears and the assumption (F4) is classical for such methods.

In this setting, neither the Palais-Smale nor Cerami compactness condition is satisfied, so compactness must be recovered by diffrent means. The classical method for obtaining strong convergence of the Cerami and Palais-Smale sequences is the so-called {\em profile decomposition,} which was introduced by G\'erard \cite{Gerard} and Nawa \cite{Nawa}. However, these theorems cannot be applied in our problem due to the fact that singular potential is not translation-invariant. Therefore, inspired by~\cite{BBDS}, we establish a new profile decomposition adapted to the non-translation-invariant structure of $X^1(\R^N)$. This part is particularly technically complicated since we show that any possible weak limit of a translated Cerami sequence must be a critical point of the limiting functional (without the singular potential) defined only on $H^1(\R^N)$.
By comparing energies as in \cite{GuoMederski}, we exclude such nontrivial limits, which ultimately forces strong convergence.

In what follows through the paper we assume the following conditions (the notation $A \lesssim B$ should be understood as $A \leq C \cdot B$, for some suitable constant $C$ which is independent of $A$ and $B$):
\begin{enumerate}
    \item[(V)] $V\in L^\infty(\R^N),$ $V$ is $\Z^N$-periodic in $x\in\R^N$ and $V_0 := \essinf_{x \in \R^N} V(x) > 0$.
    \item[(F1)] $f\colon \R^N\times\R\to\R$ is Carath\'{e}odory function that is $\Z^N$-periodic in $x\in\R^N$ and satisfies
    $$
    |f(x,u)|\lesssim 1+|u|^{p-1} \hbox{ for all } u\in\R,x\in\R^N
    $$
    for some $p \in \left(2, 2^* \right)$.
    \item[(F2)]$f(x,u)=o(u)$ as $u\to0$ uniformly with respect to $x$.
    \item[(F3)] $F(x,u)/u^2\to\infty$ uniformly in $x$ as $|u|\to\infty$, where $F$ is the usual antiderivative of $f$ with respect to $u$, namely $F(x,u)=\int_0^u f(x,s)\,ds$.
    \item[(F4)] $u\mapsto \frac{f(x,u)}{|u|}$ is nondecreasing on $(-\infty,0)$ and on $(0,\infty)$.
\end{enumerate}

Our main result is the following theorem.

\begin{Th}\label{th:main}
Suppose that (V), (F1)--(F4) hold. Then, there exists a ground state solution to \eqref{eq:main}.
\end{Th}

The paper is organized as follows. In Section~2 we introduce the functional framework associated with \eqref{eq:main} and recall several preliminary results in the space $X^1(\R^N).$ Section~3 is devoted to the variational construction of Cerami sequences. In Section~4 we establish the new profile decomposition adapted to the structure of $X^1(\R^N).$ Finally, in Section~5 we complete the proof of Theorem \ref{th:main} by showing the strong convergence of the Cerami sequence and verifying that the limit is a ground state.

In what follows, $|\cdot|_k$ denotes the usual $L^k$-norm.

\section{Variational and functional setting}
In view of the Hardy inequality
\begin{equation}\label{ineq:Hardy}
    \int_{\R^N} |\nabla u|^2 \, dx\geq \frac{(N-2)^2}{4}\int_{\R^N}\frac{u^2}{|x|^2} \, dx,
\end{equation}
with optimal constant $\frac{(N-2)^2}{4}$, one can think of equipping $H^1(\R^N)$ with equivariant norm as in \cite{GuoMederski}. However as the optimal constant in \eqref{ineq:Hardy} appears in \eqref{eq:main}, it is not possible, i.e. we loose completeness of the space. Having that in mind as an energy space we choose
$$
X^1(\R^N):=\overline{H^1(\R^N)},
$$
where the closure is taken with respect to the norm
$$
\|u\|^2_{X^1}:=\int_{\R^N} \left( |\nabla u|^2-\frac{(N-2)^2}{4|x|^2}u^2+u^2 \right)\,dx.
$$
The norm is then characterized by (see \cite{TrachanasZographopoulos}) 
\begin{align*}
\|u\|^2_{X^1} &=[u]^2 + |u|_2^2, \\
[u]^2 &:= \lim_{\eps\to 0^+}\left(\int_{|x|>\eps} |\nabla u|^2 \, dx - \int_{|x| > \varepsilon}\frac{(N-2)^2}{4|x|^2}u^2\,dx-\frac{N-2}{2} \eps^{-1} \int_{|x|=\eps}u^2\,dS\right).
\end{align*}
Moreover, $X^1 (\R^N)$ is a Hilbert space, with the scalar product given by
\begin{multline*}
\langle u, v \rangle_{X^1} := \\
\lim_{\eps\to 0^+}\left(\int_{|x|>\eps} \nabla u \nabla v \, dx - \int_{|x| > \varepsilon}\frac{(N-2)^2}{4|x|^2}uv\,dx-\frac{N-2}{2} \eps^{-1} \int_{|x|=\eps}uv\,dS\right) + \int_{\R^N} uv \, dx.
\end{multline*}
From \cite{Frank}, we have the following chain of inclusions:
\begin{equation}\label{eq:embeddings}
H^1(\R^N)\hookrightarrow X^1 (\R^N) \hookrightarrow H^s(\R^N) \hbox{ for every }s\in[0,1)
\end{equation}
In particular,
$$
X^1 (\R^N) \hookrightarrow L^t (\R^N)
$$
for $t \in [2,2^*)$, and
$$
X^1 (\R^N) \hookrightarrow \hookrightarrow L^t_{\mathrm{loc}} (\R^N)
$$
for $t \in [2, 2^*)$.

For our purposes, we prefer to introduce a new norm
\begin{equation}\label{norm_formula}
\|u\|^2 := [u]^2 + \int_{\R^N} V(x) u^2 \, dx,
\end{equation}
which, thanks to (V), is equivalent with $\| \cdot \|_{X^1}$. Moreover, by $\langle \cdot, \cdot \rangle$ we will denote the corresponding scalar product on $X^1 (\R^N)$. Let us define the energy functional $\cJ \colon X^1(\R^N)\to\R$ by
$$
\cJ(u):= \frac12 \|u\|^2 -\int_{\R^N}F(x,u)\,dx,
$$
which is of class $\cC^1$ and its critical points correspond to weak solutions of \eqref{eq:main}, see e.g. \cite{BBDS}. 

Before proceeding, we record the following elementary inequalities, see for instance \cite{SzulkinWeth}.
\begin{Lem}\label{ineq:f}
Under assumptions (F1)--(F4), 
$$
    f(x,u)u\geq 2 F(x,u)\geq 0.
$$
Moreover, for every $\eps>0$ there exist $C_\eps>0$ such that
$$
|f(x,u)|\leq \eps |u|+C_\eps |u|^{p-1}.
$$
\end{Lem}

\section{Critical point theory}
We review an abstract theoerem proved in \cite{BB}. For a $\cC^1$-functional $\cJ$ defined on a Hilbert space $E$, we define the Nehari manifold by
$$
\cN := \left\{ u \in E \setminus \{0\} : \cJ'(u)(u) = 0 \right\},
$$
which contains all nontrivial critical points of $\cJ$.

\begin{Th}[{\cite[Theorem 5.1]{BB}}]\label{Thm:Crtical_point}
    Let $E$ be a Hilbert space. Let the functional $\cJ\colon E\to\R$ be defined by
    $$
    \cJ(u):=\frac12\|u\|^2-\cI(u),
    $$
    where $\cI \colon E \rightarrow \R$ is of $\cC^1$ class. Suppose that 
    \begin{enumerate}
        \item[(J1)] there exists a constant $r>0$ such that
            $$
            \inf_{\|u\|=r}\cJ(u)>0=\cJ(0),
            $$
        \item[(J2)] $\frac{\cI(t_nu_n)}{t_n}\to \infty$ whenever $t_n\to\infty$ in $\mathbb{R}$ and $u_n\to u\neq 0$ in $E$,
        \item[(J3)] for all $t>0$ and $u\in\cN,$ there holds 
        $$
        \frac{t^2-1}{2}\cI'(u)(u)-\cI(tu)+\cI(u)\leq 0.
        $$
\end{enumerate}
Then $\cN \neq \emptyset$ and
\begin{equation}\label{ineq:crit_pt_thm}
0 < \inf_{\|u\|=r} \cJ(u) \leq c :=  \inf_{\cN} \cJ = \inf_{u \in E \setminus \{0\}} \sup_{t \geq 0} \cJ(tu).
\end{equation}
Moreover there exists a Cerami sequence for $\cJ$ on the level $c$, i.e. a sequence $(u_n)_n \subset E$ such that
$$
\cJ(u_n) \to c, \quad (1+\|u_n\|) \cJ'(u_n) \to 0.
$$
\end{Th}

To use this abstract result in our setting, we choose $E := X^1 (\R^N)$ and $\cI(u) := \int_{\R^N} F(x,u) \, dx$. In the next lemma we verify the validity of assumptions (J1)--(J3) in our case.

\begin{Lem}
There exists $r>0$ such that 
    $$
    \inf_{ u\in X^1(\R^N), \ \|u\|=r}\cJ(u)>0.
    $$

\end{Lem}
\begin{proof}
For each $u\in X^1 (\R^N)$ we use \eqref{ineq:f} and \eqref{eq:embeddings} to compute
\begin{align*}
    \cJ(u)&=\frac12 \|u\|^2 -\int_{\R^N}F(x,u)\,dx\\
    &\geq \frac12 \|u\|^2-\eps |u|_2^2-C_\eps|u|^p_p\gtrsim \frac12 \|u\|^2-\eps \|u\|^2-C_\eps\|u\|^p.
\end{align*}
Hence, for small $\eps>0$ we can choose $r>0$ such that $\cJ(u)>0$ when $\|u\|=r$.
\end{proof}

\begin{Lem}
Suppose that $t_n \to \infty$ and $(u_n)_n \subset X$, $u_n \to u \neq 0$ in $X^1 (\R^N)$. Then
$$
\int_{\R^N} \frac{F(t_n u_n)}{t_n^2} \, dx \to \infty.
$$
\end{Lem}

\begin{proof}
From (F3) and Fatou's lemma
$$
\int_{\R^N} \frac{F(t_n u_n)}{t_n^2} \, dx = \int_{\R^N} \frac{F(t_n u_n)}{t_n^2 u_n^2} u_n^2 \, dx \geq \int_{\supp u} \frac{F(t_n u_n)}{t_n^2 u_n^2} u_n^2 \, dx \to \infty
$$
as $n \to \infty$.
\end{proof}

\begin{Lem}\label{ineq:J(tu)}
For every $u \in X^1 (\R^N)$,
$$
\cJ(u) \geq \cJ(tu) - \frac{t^2-1}{2} \cJ'(u)(u) \hbox{ for } t \in (0,\infty).
$$
In particular, $\cJ(u) \geq \cJ(tu)$ for $t \in (0,\infty)$ and $u \in \cN$, and (J3) holds.
\end{Lem}
\begin{proof}
    The proof follows from calculations in \cite{dePaiva, SzulkinWeth, BieganowskiMederski} with minor modifications. To sketch the proof we introduce 
    $$
    \psi(t):= \cJ(u) - \cJ(tu) + \frac{t^2-1}{2} \cJ'(u)(u) \hbox{ for } t \in (0,\infty)
    $$
    and note that 
    $$
    \psi'(t)=-\cJ'(tu)(u)+t\cJ'(u)(u)=\int_{\R^N} f(x,tu)u-f(x,u)tu\,dx.
    $$
    From (F4) $\psi'(t)\geq0$ for $t>1$ and $\psi'(t)\leq0$ for $t<1.$ Hence, $\psi$ has a global minimum at $t=1$ and 
    $$
    \psi(t)\geq \psi(1)=0\quad \hbox{for } t>0.
    $$
\end{proof}

\section{Analysis of Cerami sequences}
We recall the following variant of the Lions' concentration-compactness principle.
\begin{Lem}[{\cite[Lemma 4.4]{BBDS}}]\label{lions}
    Let $R>0$ and $(u_n)_n \subset X^1(\R^N)$ be a bounded sequence. If
    $$
    \lim_{n\to\infty}\sup_{y\in\R^N} \int_{B(y,R)} u_n^2\,dx=0
    $$
    then $u_n\to 0$ in $L^p(\R^N)$ for $p\in (2, 2^*)$
\end{Lem}

\begin{Lem}\label{cerami-bounded}
    Any Cerami sequence $(u_n)_n \subset H^1 (\R^N)$ is bounded in $X^1 (\R^N)$.
\end{Lem}

\begin{proof}
    Up to subsequences, we suppose on the contrary that $\|u_n\|\to\infty$, and we define 
    $$v_n:=\frac{u_n}{\|u_n\|}.$$
    Without loss of generality we may assume that
    $$v_n\weakto v \hbox{ in } X^1(\R^N), \ v_n\to v\hbox{ a.e. in }\R^N.$$ 
    We claim that 
    \begin{equation}\label{lionsCondition}
    \lim_{n\to \infty} \sup_{y\in\R^N}\int_{B(y,r)} |v_n|^2\,dx>0.
    \end{equation}
    Indeed, if the left-hand side is zero, then $v_n\to 0$ in $L^p(\R^N)$ from Lions' lemma (Lemma \ref{lions}). In view of Lemma \ref{ineq:f}, we get that $\int_{\R^N}F(x,sv_n)\,dx\to 0$ for any $s>0$. Let us fix $s>0$ and observe that Lemma \ref{ineq:J(tu)} implies
$$
\cJ(u_n)\geq \cJ(sv_n)+o(1).
$$
Hence
$$
1 \gtrsim \limsup_{n\to\infty} \cJ(u_n)\geq \limsup_{n\to\infty}\cJ(sv_n)\geq \limsup_{n\to\infty}\frac{s^2}{2}\|v_n\|^2 = \frac{s^2}{2}.
$$
Since $s>0$ is arbitrary we get contradiction. Hence \eqref{lionsCondition} is proved.
As a consequence, we can find a sequence $(y_n)_n\subset \Z^N$ and radius $r>0$ such that 
$$
\liminf_{n\to \infty} \int_{B(y_n,r)} |v_n|^2\,dx = \liminf_{n\to \infty} \int_{B(0,r)} |v_n(\cdot + y_n)|^2\,dx >0.
$$
Therefore, up to a subsequence, $v_n(\cdot+y_n)\to v$ in $L^2_{\loc}(\R^N)$ and $v\neq 0$. If $v(x)\neq 0,$ then $|u_n(x+y_n)|=|v_n(x+y_n)|\|u_n\|\to\infty$. Using (F3)
$$
\frac{F(x,u_n(x+y_n))}{\|u_n\|^2}=\frac{F(x, u_n(x+y_n))}{|u_n(x+y_n)|^2}|v_n(x+y_n)|^2\to \infty.
$$
Then by the Fatou's lemma we get
\begin{align*}
    0=\limsup_{n\to\infty} \frac{\cJ(u_n)}{\|u_n\|^2}=\limsup_{n\to\infty} \left(\frac12-\int_{\R^N}\frac{F(x,u_n(x+y_n)}{\|u_n\|^2}\right)=-\infty,
\end{align*}
which is again a contradiction.
\end{proof}

Let us define the functional $\cJ_0\colon H^1(\R^N)\to \R$ with the formula
$$
\cJ_0(u)=\frac12\int_{\R^N}|\nabla u|^2+V(x)u^2-F(x,u)\,dx.
$$
Now we are able to prove the profile decomposition of Cerami sequences that are in $H^1 (\R^N)$. Since $H^1(\R^N)$ is a dense subspace, we can replace any Cerami sequence in $X^1(\R^N)$ by a Cerami sequence in $H^1(\R^N)$. Therefore, this additional assumption is not restrictive, but it simplifies a bit the proof, since $H^1 (\R^N)$ is the maximal, translation-invariant subspace of $X^1(\R^N)$, see \cite{BBDS}.

The following regularity result will be used in the proof of our decomposition theorem for Cerami sequences. We recall its proof for the reader's sake.
\begin{Lem} \label{lem:4.3}
    Let $\Omega$ be an open subset of $\R^N$. If $u$ is a distribution on $\Omega$ such that $\Delta u \in H^{-1}(\Omega)$, then $u \in H_{\mathrm{loc}}^{1}(\Omega)$.
\end{Lem}
\begin{proof}
    We follow \cite{DiFrattaFiorenza}. Let $v \in H_0^{1}(\Omega)$ be the unique solution to the equation $\Delta v = \Delta u$ in the sense of distributions. Hence $\Delta (u-v)=0$ in the sense of distributions, and Weyl's Lemma (see \cite[Theorem 8.12]{Rudin}) implies that $u-v \in \cC^\infty(\Omega)$. It follows that $u=(u-v)+v$ is the sum of a smooth function and of a function from $H_0^1(\Omega)$. In particular, $u \in H_{\mathrm{loc}}^{1}(\Omega)$. 
\end{proof}

\begin{Th}
    Let $(u_n)_n\subset H^1(\R^N)$ be a bounded in $X^1(\R^N)$ Cerami sequence of $\cJ$ at level $c\geq 0$. Then there are $m\geq 0$, $\overline{u}_0 \in X^1 (\R^N)$, and there are sequences $(\overline{u}_i)_{i=1}^m\subset H^1 (\R^N)$ and $(x^i_n)_{0\leq i\leq m}\subset\Z^N$ such that $x^0_n=0, |x^i_n|\to\infty,|x^i_n-x^j_n|\to\infty, i\neq j, i,j=1,\ldots,m$ and passing to a subsequence the following hold:
\begin{align}
   \label{a} &\cJ'(\overline{u}_0)=0;\\
    \label{c}&\overline{u}_i \neq 0 \hbox{ and } \cJ_0'(\overline{u}_i) = 0 \hbox{ for each } 1\leq i\leq m;\\
   \label{d} & \left\|u_n-\overline{u}_0-\sum_{i=0}^m\overline{u}_i\right\| \underset{n\to\infty}\longrightarrow 0\\
   \label{e} &\cJ(u_n)\to \cJ(\overline{u}_0)+\sum_{i=1}^m\cJ_0(\overline{u}_i)
\end{align}
    
\end{Th}
\begin{proof}
The proof is divided into several steps.
\begin{enumerate}
    \item[\bf Step 1.] Since $(u_n)_n$ is bounded in $X^1 (\R^N)$, there exists $\overline{u}_0\in X^1(\R^N)$ such that, up to a subsequence, 
    $$
    u_n\weakto \overline{u}_0 \hbox{ in } X^1 (\R^N), \ u_n\to \overline{u}_0 \hbox{ in } L^p_{\loc} (\R^N), \hbox{ and } u_n\to \overline{u}_0 \hbox{ a.e. in }\R^N.
    $$
        Then for every $\varphi\in C^\infty_0(\R^N)$
    $$\cJ'(u_n)(\varphi)-\cJ'(\overline{u}_0)(\varphi)=\langle u_n-\overline{u}_0,\varphi \rangle-\int_{\R^N}\left(f(x,u_n)-f(x,\overline{u}_0)\right)\varphi\,dx\to 0,$$
    where we recall that $\langle\cdot,\cdot\rangle$ is the scalar product corresponding to \eqref{norm_formula}, $\langle u_n-\overline{u}_0,\varphi\rangle\to 0$ is a consequence of a weak convergence and the convergence $\int_{\R^N}\left(f(x,u_n)-f(x,\overline{u}_0)\right)\varphi\,dx\to 0$ follows from the Vitali convergence theorem. Therefore $\cJ'(\overline{u}_0) = 0$.
    \item[\bf Step 2.]{\em Let $v_n^1:=u_n-\overline{u}_0$. Suppose that 
    $$
    \sup_{z\in\R^N} \int_{B(z,1)}|v_n^1|^2\,dx\to 0.
    $$
    Then we shall show $u_n\to \overline{u}_0$ in $X^1 (\R^N)$ and the statement holds for $m=0$.}\\
    Note that
    \begin{align*} 
    \cJ'(u_n)(v_n^1) &=\langle u_n, v_n^1 \rangle-\int_{\R^N}f(x,u_n)v_n^1\,dx \\
    &= \| v_n^1 \|^2 + \langle \overline{u}_0, v_n^1 \rangle -\int_{\R^N}f(x,u_n)v_n^1\,dx.
    \end{align*}
    Therefore
    $$
    \|v_n^1\|^2=\cJ'(u_n)(v_n^1)-\langle \overline{u}_0, v_n^1 \rangle+\int_{\R^N}f(x,u_n)v_n^1\,dx.
    $$
    Since $0 = \cJ'(\overline{u}_0)(v_n^1) = \langle \overline{u}_0, v_n^1 \rangle - \int_{\R^N} f(x, \overline{u}_0) v_n^1 \, dx$, we obtain
     $$
    \|v_n^1\|^2=\cJ'(u_n)(v_n^1)+\int_{\R^N}(f(x,u_n)-f(x,\overline{u}_0))v_n^1\,dx.
    $$
    Since $(v_n^1)_n$ is bounded, 
    $$\|\cJ'(u_n)(v_n^1)\|\leq \|\cJ'(u_n)\|\|v_n^1\|\to 0.$$
    From the Lions' lemma we get that $v_n^1\to 0$ in $L^p(\R^N)$.
    Observe that for every $\varepsilon > 0$, there exists a positive constant $C_\varepsilon$ such that
    \begin{align*}
    \limsup_{n\to\infty} \left| \int_{\R^N} f(x, \overline{u}_0) v_n^1 \, dx \right| &\leq \limsup_{n\to\infty} \left( \varepsilon \int_{\R^N} |\overline{u}_0 v_n^1| \, dx + C_\varepsilon \int_{\R^N} |\overline{u}_0|^{p-1} |v_n^1| \, dx \right) \\
    &\leq \limsup_{n\to\infty} \left( \varepsilon |\overline{u}_0|_2 |v_n^1|_2 + C_\varepsilon |\overline{u}_0|_p^{p-1} |v_n^1|_p \right) \\
    &= \varepsilon |\overline{u}_0|_2 \limsup_{n\to\infty} |v_n^1|_2.
    \end{align*}
    Since $\varepsilon > 0$ was arbitrary, $\int_{\R^N} f(x, \overline{u}_0) v_n^1 \, dx \to 0$. Similarly $\int_{\R^N} f(x, u_n) v_n^1 \, dx \to 0$. Hence $v_n^1 \to 0$ in $X^1 (\R^N)$.

\item[] From now on, we fix $s \in \left( \max \left\{\frac12, \frac{(p-2)N}{2p} \right\}, 1 \right)$.
    
    \item[\bf Step 3.] {\em Suppose that $(z_n)_n\subset \Z^N$ such that
    \begin{equation}\label{non-vanishing}
    \liminf_{n\to\infty}\int_{B(z_n,1+\sqrt{N})}|v_n^1|^2\,dx>0.
    \end{equation}
    Then there is $\overline{u}\in H^1(\R^N)$ such that (up to a subsequence)}
    $$
    \hbox{\textit{(i)} } |z_n|\to\infty, \hbox{ \textit{(ii)} }u_n(\cdot+z_n)\weakto\overline{u}\neq 0 \mbox{ in } H^s(\R^N), \hbox{ \textit{(iii)} } \cJ_0'(\overline{u})=0.
    $$
    Condition (i) is standard. Since $u_n \in H^1 (\R^N) \subset X^1 (\R^N)$ and $H^1 (\R^N)$ is a translation-invariant subspace, $u_n (\cdot + z_n) \subset H^1 (\R^N) \subset X^1 (\R^N)$. Since $X^1 (\R^N) \subset H^s (\R^N)$, $(u_n)$ is bounded in $H^s (\R^N)$ and therefore $(u_n(\cdot + z_n))$ is also bounded in $H^s (\R^N)$. Hence, up to a subsequence, 
    \begin{align*}
    &u_n(\cdot + z_n) \weakto \overline{u} \hbox{ in }H^s (\R^N), \\ 
    &u_n(\cdot + z_n) \to \overline{u} \hbox{ in } L^p_{\loc} (\R^N), \\
    &u_n(x + z_n) \to \overline{u}(x) \hbox{ a.e. in } \R^N.
    \end{align*}
    From \eqref{non-vanishing} it is clear that $\overline{u} \neq 0$. Now we will show (iii). Denote $v_n:=u_n(\cdot+z_n)$.
    Take any $\varphi \in C_0^\infty (\R^N)$ and note that
    $$
    \|\cJ'(u_n)(\varphi(\cdot-z_n))\|\leq \|\cJ'(u_n)\|\|\varphi(\cdot-z_n)\| \to 0
    $$
    thanks to \cite[Lemma 4.3]{BBDS}. Therefore
    \begin{align*}
    o(1) &= \langle u_n, \varphi(\cdot - z_n) \rangle -\int_{\R^N}f(x,u_n)\varphi(\cdot - z_n)\,dx \\
    &= \int_{\R^N}\nabla v_n\nabla\varphi-\frac{(N-2)^2}{4|x+z_n|^2}v_n\varphi\,dx+\int_{\R^N} V(x)v_n\varphi\,dx-\int_{\R^N}f(x,v_n)\varphi\,dx \\
    &= -\int_{\R^N} v_n\Delta\varphi-\frac{(N-2)^2}{4|x+z_n|^2}v_n\varphi\,dx+\int_{\R^N} V(x)v_n\varphi\,dx-\int_{\R^N}f(x,v_n)\varphi\,dx.
    \end{align*}
    Note that $\kappa_n := \inf_{x \in \supp \varphi} |x+z_n|^2 \to \infty$ and therefore
    $$
    \left| \int_{\R^N} \frac{v_n \varphi}{|x+z_n|^2} \, dx \right| \leq \frac{1}{\kappa_n} |v_n|_2 |\varphi|_2 \to 0,
    $$
    since $(v_n)_n$ is bounded in $L^2 (\R^N)$. Now, using that $v_n \to \overline{u}$ in $L^2_{\loc} (\R^N)$ and in $L^p_{\loc} (\R^N)$, we get that
    \begin{align*}
    &\quad -\int_{\R^N} v_n\Delta\varphi-\frac{(N-2)^2}{4|x+z_n|^2}v_n\varphi\,dx+\int_{\R^N} V(x)v_n\varphi\,dx-\int_{\R^N}f(x,v_n)\varphi\,dx \\
    &\to -\int_{\R^N} \overline{u} \Delta\varphi\,dx+\int_{\R^N} V(x)\overline{u}\varphi\,dx-\int_{\R^N}f(x,\overline{u})\varphi\,dx.
    \end{align*}
    Therefore $\overline{u}$ satisfies
    \begin{equation}\label{eq:limiting}
    -\Delta \overline{u} + V(x) \overline{u} = f(x, \overline{u}) \quad \mbox{in } \cD'(\R^N)
    \end{equation}
    or, equivalently,
    $$
    -\Delta \overline{u}  = f(x, \overline{u}) - V(x) \overline{u} =: h(x) \quad \mbox{in } \cD'(\R^N).
    $$
    In view of Lemma \ref{lem:4.3}, to show that $\overline{u} \in H^1_{\loc} (\R^N)$, it is enough to verify that $h \in H^{-1} (\R^N)$. Observe that
    $$
    |h| \leq |V|_\infty |\overline{u}| + |\overline{u}| + C_1 |\overline{u}|^{p-1} \in L^2 (\R^N) + L^{\frac{p}{p-1}} (\R^N) \subset H^{-1}(\R^N).
    $$
    Hence, indeed $\overline{u} \in H^1_{\loc} (\R^N)$. To see that $\overline{u} \in H^1 (\R^N)$, we show a variant of the Caccioppoli inequality. Take any $R > 0$ and consider a function $\eta_R \in C_0^\infty (B(0,2R))$ such that $\eta_R \equiv 1$ on $B(0,R)$, $0 \leq \eta_R \leq 1$ and $| \nabla \eta_R | \leq \frac{c}{R}$ for some $c \geq 1$. Then, testing \eqref{eq:limiting} with $\eta_R^2 \overline{u}$ we get
    $$
    \int_{\R^N} \eta_R^2 |\nabla \overline{u}|^2 \, dx + 2 \int_{\R^N} \eta_R \overline{u} \nabla \overline{u} \nabla \eta_R \, dx + \int_{\R^N} V(x) \overline{u}^2 \eta_R^2 \, dx = \int_{\R^N} f(x,\overline{u}) \eta_R^2 \overline{u} \, dx.
    $$
    Using Young's inequality
    $$
    \left| 2 \int_{\R^N} \eta_R \overline{u} \nabla \overline{u} \nabla \eta_R \, dx \right| \leq \frac12 \int_{\R^N} \eta_R^2 |\nabla \overline{u}|^2 + \frac{2c^2}{R^2} \int_{B(0,2R)\setminus B(0,R)} \overline{u}^2 \, dx.
    $$
    Hence
    \begin{align*}
    \int_{B(0,R)} |\nabla \overline{u}|^2 \, dx &\leq \int_{\R^N} \eta_R^2 |\nabla \overline{u}|^2 \, dx \\
    &\lesssim \int_{\R^N} |f(x,\overline{u}) \overline{u}| \eta_R^2 \, dx + \int_{\R^N} |V(x)| \overline{u}^2 \eta_R^2 \, dx + \frac{1}{R^2} \int_{B(0,2R)\setminus B(0,R)} \overline{u}^2 \, dx \\
    &\leq \int_{\R^N} |f(x,\overline{u}) \overline{u}|  \, dx + |V|_\infty \int_{\R^N} \overline{u}^2  \, dx + \frac{1}{R^2} \int_{\R^N} \overline{u}^2 \, dx.
    \end{align*}
    From monotone convergence theorem, $\nabla \overline{u} \in L^2 (\R^N)$, and $\overline{u} \in H^1 (\R^N)$. In particular, $\cJ_0'(\overline{u}) = 0$.

    \item[\bf Step 4.]
    {\em
    Suppose there exist $m\geq 1$, $(x_n^i)_n \subset\Z^N, \overline{u_i}\in H^1(\R^N)$ for $1\leq i\leq m$ such that 
    \begin{align*}
        |x^i_n|\to\infty,|x^i_n-x^j_n|\to\infty\hbox{ for }i\neq j,\\
        u_n(\cdot+x^i_n)\to\overline{u_i}, \hbox{ for each } 1\leq i\leq m,\\
        \cJ'_0(\overline{u_i})=0 \hbox{ for each } 1\leq i\leq m.
    \end{align*}
    Then
    \begin{enumerate}
        \item if $\sup_{z\in\R^N}\int_{B(z,1)}|u_n-\overline{u}_0-\sum_{i=1}^m\overline{u_i}(\cdot-x^i_n)|^2\,dx\to 0$ as $n\to\infty,$ then
        $$
        \left\|u_n-\overline{u}_0-\sum_{i=1}^m\overline{u_i}(\cdot-x^i_n)\right\|\to 0;
        $$
      \item if there exist $(z_n)\subset\Z^N$ such that
        $$
        \liminf_{n\to\infty}\int_{B(z_n,1+\sqrt{N})}\left|u_n-\overline{u}_0-\sum_{i=1}^m\overline{u}_i(\cdot-x^i_n)\right|^2\,dx>0,
        $$
        then there is $\overline{u}_{m+1}$ such that up to a subsequence
        \begin{enumerate}
            \item $|z_n|\to\infty,|z_n^i-z_n^j|\to\infty,$ for $1\leq i\leq m,$
            \item $u_n(\cdot+z_n)\weakto \overline{u}_{m+1}\neq 0,$ in $H^s(\R^N).$
            \item $\cJ'_0(\overline{u})=0.$
        \end{enumerate}
    \end{enumerate}
    }
    Suppose that $\sup_{z\in\R^N}\int_{B(z,1)}|u_n-\overline{u}_0-\sum_{i=1}^m\overline{u_i}(\cdot-x^i_n)|^2\,dx\to 0$ as $n\to\infty.$ Let
    $$
    \xi_n=u_n-\overline{u}_0-\sum_{i=1}^m\overline{u_i}(\cdot-x^i_n) \in X^1(\R^N)
    $$
    and note that from the Lions lemma we get that $\xi_n\to 0$ in $L^r(\R^N)$, $2 < r <2^*$. We compute
    \begin{align*}
    \cJ'(u_n)(\xi_n)&=\|\xi_n\|^2 + \langle \overline{u}_0, \xi_n\rangle + \sum_{i=1}^m \langle \overline{u_i}(\cdot-x^i_n), \xi_n\rangle -\int_{\R^N}f(x,u_n)\xi_n\,dx,
     \end{align*}
     and since $\cJ'(\overline{u}_0)(\xi_n)=0$ we get that
     \begin{align*}
     \|\xi_n\|^2=\cJ'(u_n)(\xi_n)-\int_{\R^N}f(x,\overline{u}_0)\xi_n\,dx-\sum_{i=1}^m \langle \overline{u_i}(\cdot-x^i_n), \xi_n\rangle +\int_{\R^N}f(x,u_n)\xi_n\,dx.
     \end{align*}
     Observe that 
     \begin{align*} 
     \|\cJ'(u_n)(\xi_n)\| &\leq \|\cJ'(u_n)\|\|\xi_n\| \\
     &\leq \|\cJ'(u_n)\| \left( \|u_n\| + \|\overline{u}_0\| + \sum_{i=1}^m \| \overline{u_i}(\cdot-x^i_n) \| \right)\to 0
     \end{align*}
     thanks to \cite[Lemma 4.3]{BBDS}, and since $|\xi_n|_r\to 0$, 
     $$
     \int_{\R^N}f(x,\overline{u}_0)\xi_n\,dx \to 0, \quad \int_{\R^N}f(x,u_n)\xi_n\,dx \to 0.
     $$
     Hence
     \begin{align*}
     \|\xi_n\|^2 &= -\sum_{i=1}^m \langle \overline{u_i}(\cdot-x^i_n), \xi_n\rangle + o(1) \\
     &= -\sum_{i=1}^m \left( \langle \overline{u_i}(\cdot-x^i_n), u_n \rangle - \langle \overline{u_i}(\cdot-x^i_n), \overline{u}_0 \rangle - \sum_{j=1}^m \langle \overline{u_i}(\cdot-x^i_n), \overline{u_j}(\cdot-x^j_n) \rangle \right) + o(1).
     \end{align*}
     From \cite[Lemma 4.1]{BBDS}, $\langle \overline{u_i}(\cdot-x^i_n), \overline{u}_0 \rangle \to 0$. Thus
     \begin{align*}
     \|\xi_n\|^2 &= -\sum_{i=1}^m \left( \langle \overline{u_i}(\cdot-x^i_n), u_n \rangle  - \sum_{j=1}^m \langle \overline{u_i}(\cdot-x^i_n), \overline{u_j}(\cdot-x^j_n) \rangle \right) + o(1)\\
     &= - \sum_{i=1}^m \left( \langle \overline{u}_i, u_n (\cdot + x_n^i) \rangle - \sum_{j=1}^m \left\langle \overline{u}_i, \overline{u}_j \left( \cdot - (x_n^j - x_n^i) \right) \right\rangle \right) \\
     &\quad + \frac{(N-2)^2}{4} \sum_{i=1}^m \left( \int_{\R^N} \frac{\overline{u}_i (\cdot - x_n^i) u_n}{|x|^2} \, dx - \int_{\R^N} \frac{\overline{u}_i  u_n (\cdot + x_n^i)}{|x|^2} \, dx\right) \\
     &\quad + \frac{(N-2)^2}{4} \sum_{i=1}^m \sum_{j=1}^m \left( - \int_{\R^N} \frac{\overline{u_i}(\cdot-x^i_n)\overline{u_j}(\cdot-x^j_n)}{|x|^2} \, dx + \int_{\R^N} \frac{\overline{u}_i \overline{u}_j \left( \cdot - (x_n^j - x_n^i) \right)}{|x|^2} \, dx \right) + o(1).
     \end{align*}
     Again, \cite[Lemma 4.1]{BBDS} implies that $\left\langle \overline{u}_i, \overline{u}_j \left( \cdot - (x_n^j - x_n^i) \right) \right\rangle \to 0$ for $i \neq j$. Moreover, using \cite[Lemma 3.1]{GuoMederski} and H\"older inequality
     \begin{align*}
     \left| \int_{\R^N} \frac{\overline{u}_i \overline{u}_j \left( \cdot - (x_n^j - x_n^i) \right)}{|x|^2} \, dx \right| \leq \left( \int_{\R^N} \frac{|\overline{u}_i|^2}{|x|^2} \, dx \right)^{1/2} \left( \int_{\R^N} \frac{| \overline{u}_j \left( \cdot - (x_n^j - x_n^i) \right) |^2}{|x|^2} \, dx  \right)^{1/2} \to 0
     \end{align*}
     for $i \neq j$. Also, for any $i$, $j$, we have
     $$
      \left| \int_{\R^N} \frac{\overline{u_i}(\cdot-x^i_n)\overline{u_j}(\cdot-x^j_n)}{|x|^2} \, dx \right| \leq \left( \int_{\R^N} \frac{|\overline{u_i}(\cdot-x^i_n)|^2}{|x|^2} \, dx \right)^{1/2} \left( \int_{\R^N}\frac{|\overline{u_j}(\cdot-x^j_n)|^2}{|x|^2} \, dx \right)^{1/2} \to 0.
     $$
     Therefore
     \begin{align*}
     \|\xi_n\|^2 &= - \sum_{i=1}^m \left( \int_{\R^N} \nabla \overline{u}_i \nabla \left( u_n (\cdot + x_n^i) - \overline{u}_i \right) \, dx + \int_{\R^N}V(x) \overline{u}_i \left( u_n (\cdot + x_n^i) - \overline{u}_i \right) \, dx \right)  \\
     &\quad + \frac{(N-2)^2}{4} \sum_{i=1}^m \left( \int_{\R^N} \frac{\overline{u}_i (\cdot - x_n^i) u_n}{|x|^2} \, dx \right) + o(1).
     \end{align*}
    Using that $\cJ_0'(\overline{u}_i) \left( u_n (\cdot + x_n^i) - \overline{u}_i \right) = 0$, we obtain that
    \begin{align*}
    &\quad \int_{\R^N} \nabla \overline{u}_i \nabla \left( u_n (\cdot + x_n^i) - \overline{u}_i \right) \, dx + \int_{\R^N}V(x) \overline{u}_i \left( u_n (\cdot + x_n^i) - \overline{u}_i \right) \, dx \\
    &= \int_{\R^N} f(\overline{u}_i) \left( u_n (\cdot + x_n^i) - \overline{u}_i \right) \, dx \to 0
    \end{align*}
    thanks to Vitali's convergence theorem. Now, we consider the term
    $$
    \int_{\R^N} \frac{\overline{u}_i (\cdot - x_n^i) u_n}{|x|^2} \, dx = \int_{B(0,R)} \frac{\overline{u}_i (\cdot - x_n^i) u_n}{|x|^2} \, dx + \int_{\R^N \setminus B(0,R)} \frac{\overline{u}_i (\cdot - x_n^i) u_n}{|x|^2} \, dx =: I_1 + I_2,
    $$
    where $R > 0$ is fixed. Then
    \begin{align*}
    |I_2| \leq \frac{1}{R^2} |\overline{u}_i|_2 |u_n|_2 \lesssim \frac{1}{R^2},
    \end{align*}
    since $(u_n)_n$ is bounded in $L^2 (\R^N)$. To estimate $I_1$ we use H\"older's inequality, exponential decay of $\overline{u}_i$ (\cite[Theorem 2]{Pankov}), boundedness of $(u_n)_n$ in $L^{\frac{2N}{N-2s}}(\R^N)$ and the inequality $\frac{4N}{N+2s} < N$,
    \begin{align*}
    |I_1| \leq |u_n|_{\frac{2N}{N-2s}} \left( \int_{|x| < R} \frac{|\overline{u}_i (\cdot - x_n^i)|^{\frac{2N}{N+2s}}}{|x|^{\frac{4N}{N+2s}}} \, dx \right)^{\frac{N+2s}{2N}} \lesssim \sup_{|x| < R} e^{-\alpha |x-x_n^i|} \to 0,
    \end{align*}
    for some $\alpha > 0$. Hence, for any $R > 0$,
    $$
    \left| \int_{\R^N} \frac{\overline{u}_i (\cdot - x_n^i) u_n}{|x|^2} \, dx \right| \lesssim \frac{1}{R^2} + o(1)
    $$
    and therefore
    $$
    \int_{\R^N} \frac{\overline{u}_i (\cdot - x_n^i) u_n}{|x|^2} \, dx \to 0.
    $$
    It means that $\|\xi_n\|^2 \to 0$ and $\xi_n \to 0$ in $X^1 (\R^N)$, and the proof in this case is complete.
Now assume 
\begin{equation}\label{splitting:non-vanishing}
\liminf_{n\to\infty}\int_{B(z_n,1+\sqrt{N})}\left|u_n-\overline{u}_0-\sum_{i=1}^m\overline{u}_i(\cdot-x^i_n)\right|^2\,dx>0,    
\end{equation} for some $(z_n)_n\subset\Z^N.$ Then (i) is again standard. To prove (ii) we follow as in Step 3. Since $X^1 (\R^N) \subset H^s (\R^N)$, $(u_n)$ is bounded in $H^s (\R^N)$ and therefore $(u_n(\cdot + z_n))$ is also bounded in $H^s (\R^N)$. Hence, up to a subsequence, 
    \begin{align*}
    &u_n(\cdot + z_n) \weakto \overline{u}_{m+1} \hbox{ in }H^s (\R^N), \\ 
    &u_n(\cdot + z_n) \to \overline{u}_{m+1} \hbox{ in } L^p_{\loc} (\R^N), \\
    &u_n(x + z_n) \to \overline{u}_{m+1}(x) \hbox{ a.e. in } \R^N.
    \end{align*}
    From \eqref{splitting:non-vanishing} it is clear that $\overline{u}_{m+1} \neq 0$. Similarly as in Step 3 we show that $\overline{u}_{m+1} \in H^1 (\R^N)$ and $\cJ_0'(\overline{u}_{m+1}) = 0$.

\item[\bf Step 5.] 
Step 1. implies \eqref{a}.  If first condition from Step 2. holds then $u_n\to \overline{u}_0$ and the theorem is true for $m=0.$ On the other hand if 
$$
\liminf_{n\to\infty}\int_{B(y_n,1)}|v_n^1|^2\,dx>0
$$
for $(y_n)\subset\R^N.$ For each such translation we can find $z_n$ such that $B(y_n,1)\subset B(z_n,1+\sqrt{N}).$ Then
$$
\liminf_{n\to\infty}\int_{B(z_n,1+\sqrt{N})}|v_n^1|^2\,dx\geq\liminf_{n\to\infty}\int_{B(y_n,1)}|v_n^1|^2\,dx>0.
$$
So  in view of Step 3 one can find $\overline{u}$ such that (i)-(iii) hold. Let $y^1_n=z_n$ and $\overline{u}_1:=\overline{u}.$ If (a) from Step 4 holds with $m=1,$ then \eqref{c}, \eqref{d} are true. Otherwise we put $(y_n^2):=(z_n)$ and $\overline{u}_2:=\overline{u}$ and iterate Step 4. To complete the proof of \eqref{c}, \eqref{d} it is sufficient to show that this procedure will finish. 

For this purpose, similarly as in Step 4, we compute that
\begin{align*}
 \left\| u_n - \overline{u}_0 - \sum_{i=1}^m \overline{u}_i (\cdot - x_n^i) \right\|^2 &=  \| u_n\|^2 - 2 \langle u_n, \overline{u}_0 \rangle + \|\overline{u}_0\|^2 - 2 \sum_{i=1}^m \langle u_n, \overline{u}_i (\cdot - x_n^i) \rangle  \\
&\quad + 2 \sum_{i=1}^m \langle \overline{u}_0, \overline{u}_i (\cdot - x_n^i) \rangle  + \sum_{i=1}^m \sum_{j=1}^m \langle \overline{u}_i (\cdot - x_n^i), \overline{u}_j (\cdot - x_n^j) \rangle \\
&= \|u_n\|^2 - \|\overline{u}_0\|^2 \\
&\quad- 2 \sum_{i=1}^m \left( \int_{\R^N} f(\overline{u}_i) u_n(\cdot + x_n^1) \, dx - \frac{(N-2)^2}{4} \int_{\R^N} \frac{\overline{u}_i(\cdot - x_n^i) u_n}{|x|^2} \, dx \right) \\
&\quad + \sum_{i=1}^m \left( \|\overline{u}_i (\cdot - x_n^i) \|^2 + \sum_{j \neq i} \langle \overline{u}_i (\cdot - x_n^i), \overline{u}_j (\cdot - x_n^j) \rangle \right) + o(1) \\
&= \|u_n\|^2 - \|\overline{u}_0\|^2 - 2 \sum_{i=1}^m  \int_{\R^N} f(\overline{u}_i) \overline{u}_i \, dx  \\
&\quad + \sum_{i=1}^m \int_{\R^N} |\nabla \overline{u}_i|^2 + V(x) |\overline{u}_i|^2 \, dx  + o(1) \\
&= \|u_n\|^2 - \|\overline{u}_0\|^2 - \sum_{i=1}^m \int_{\R^N} |\nabla \overline{u}_i|^2 + V(x) |\overline{u}_i|^2 \, dx + o(1)
\end{align*}
for each $m \geq 0$. Now, note that $\cJ_0'(\overline{u}_i)(\overline{u}_i) = 0$ implies that there is $\rho > 0$ such that $\int_{\R^N} |\nabla \overline{u}_i|^2 + V(x) |\overline{u}_i|^2 \, dx \geq \rho$ for each $i$. Hence
$$
m \rho \leq \|u_n\|^2 - \|\overline{u}_0\|^2 + o(1),
$$
so after a finite number of steps, (a) in Step 4 will hold.

\item[\bf Step 6.] {\em We will show that \eqref{e} holds.}

\noindent Note that, calculations in Step 5 show that
$$
\|u_n\|^2 = \|\overline{u}_0\|^2 + \sum_{i=1}^m \int_{\R^N} |\nabla \overline{u}_i|^2 + V(x) |\overline{u}_i|^2 \, dx + o(1).
$$
Hence, it is sufficient to show 
$$
\int_{\R^N} F(x, u_n) \, dx \to \int_{\R^N} F(x, \overline{u}_0) \, dx + \sum_{i=1}^m \int_{\R^N} F(x, \overline{u}_i) \, dx.
$$
From a variant of Brezis-Lieb lemma \cite{BrezisLieb}, we have that
$$
\int_{\R^N} F(x,u_n-\overline{u}_0)+F(x,\overline{u}_0)-F(x,u_n)\,dx\to 0
$$
and, then --- iterating as in \cite[Theorem 4.1]{BieganowskiMederski} --- we obtain that 
$$
\int_{\R^N} F(x, u_n) \, dx \to \int_{\R^N} F(x, \overline{u}_0) \, dx + \sum_{i=1}^m \int_{\R^N} F(x, \overline{u}_i) \, dx.
$$
\end{enumerate}
\end{proof}

\section{Proof of Theorem}

\begin{proof}
Let $(u_n)_n$ be a Cerami sequence for $\cJ$ from Theorem \ref{Thm:Crtical_point}. From the density of $H^1 (\R^N)$ in $X^1 (\R^N)$ we may assume that $(u_n)_n \subset H^1 (\R^N)$.

Denote $c := \inf_\cN \cJ$ and $c_0 := \inf_{\cN_0} \cJ_0$. Take $u_0 \in H^1 (\R^N)$ such that $\cJ_0(u_0) = c_0$, see \cite{SzulkinWeth, dePaiva}. Then. there is $t > 0$ such that $tu_0 \in \cN$. Thus, we compute
$$
c_0 = \cJ_0(u_0) \geq \cJ_0(t u_0) > \cJ (t u_0) \geq c.
$$

Then, from Lemma \ref{cerami-bounded}, $(u_n)_n$ is bounded in $X^1 (\R^N)$. From the profile decomposition we have
$$
\cJ(u_n) \to \cJ(\overline{u}_0) + \sum_{i=1}^m \cJ_0(\overline{u}_i) \geq \cJ(\overline{u}_0) + m c_0.
$$
Hence
$$
c \geq \cJ(\overline{u}_0) + m c_0.
$$
Since $c_0 > c$, we obtain $m = 0$. Hence $\cJ(\overline{u}_0) = c$ and $\overline{u}_0$ is the ground state solution to \eqref{eq:main}.
\end{proof}

\section*{Acknowledgements}

We would like to thank the anonymous reviewers for their comments, which helped us improve the quality of the paper.

Bartosz Bieganowski and Adam Konysz were partly supported by the National Science Centre, Poland (Grant No. 2022/47/D/ST1/00487).

\end{document}